\documentclass[a4paper,12pt]{article}

\usepackage{amsthm}
\usepackage{amsmath}
\usepackage{epsfig}
\usepackage{amssymb}
\usepackage{latexsym}
\usepackage{color}

\def \DG {\Delta \Gamma}

\author{Robert Lukot\!'ka, J\'{a}n Maz\'{a}k, Xuding Zhu}

\title{
Maximum $4$-degenerate subgraph of a planar graph\footnote{
The work of the first author was supported from the grant 7/TU/13 and from the APVV grants APVV-0223-10 and ESF-EC-0009-10 within the EUROCORES Programme EUROGIGA (project GReGAS) of the European Science Foundation.
The work of the second author leading to this invention has received
funding from the European Research Council under the European Union’s Seventh
Framework Programme (FP7/2007-2013)/ERC grant agreement no. 259385.
The third author acknowledges partial support from the research grants NSF No. 11171730 and ZJNSF No. Z6110786.}
}

\theoremstyle{definition}
\newtheorem{definition}{Definition}
\newtheorem{theorem}[definition]{Theorem}
\newtheorem{claim}[definition]{Claim}
\newtheorem{corollary}[definition]{Corollary}
\newtheorem{lema}[definition]{Lemma}
\newtheorem{conjecture}{Conjecture}[section]

\def\mc{\hbox{\rm mc}}
\def\ch{\hbox{\rm ch}}
\def\tc{\hbox{\rm tc}}

\def\nontr{{\bigsqcup}}
\def\K{{\cal K}}

\let\epsilon=\varepsilon

\begin{document}

\maketitle

\abstract{
A graph $G$ is \emph{$k$-degenerate} if it can be transformed into an empty graph by subsequent
removals of vertices of degree $k$ or less.
We prove that every connected planar graph with average degree $d \ge 2$  has
a $4$-degenerate induced subgraph containing at least $(38-d)/36$ of its vertices.
This shows that every planar graph of order $n$ has a $4$-degenerate induced
subgraph of order more than $8/9 \cdot n$.
We also consider a local variation of this problem and show that
in every planar graph with at least $7$ vertices, deleting a suitable vertex allows us
to subsequently remove at least $6$ more vertices of degree four or less.
}

\section{Degeneracy and choosability}

A graph $G$ is \emph{$k$-degenerate} if every subgraph of $G$ has a vertex of
degree $k$ or less. Equivalently, a graph is $k$-degenerate if we can delete
the whole graph by subsequently removing vertices of degree
at most $k$. The reverse of this sequence of removed vertices can be used to
colour (or list-colour) $G$ with $k+1$ colours in a greedy fashion. Graph
degeneracy is therefore a natural bound on both chromatic number and list chromatic number.
In certain problems, graph degeneracy even provides the best known bounds on choosability
\cite{barat}.

This article focuses on degeneracy of planar graphs. Every subgraph of a planar graph has a vertex of degree at most $5$ because it is also planar;
therefore, every planar graph is $5$-degenerate. For $k<5$, a planar graph need not to be $k$-degenerate.
  An  interesting question arises how large $k$-degenerate subgraphs in planar graphs can
be guaranteed. We discuss this question for particular values of $k$ in the following
paragraphs.

Let $G$ be a planar graph. To find a   maximum  induced $0$-degenerate subgraph, we need
to find a maximum independent set. According to the $4$-colour theorem, we can find an
independent set of order at least $1/4 \cdot |V(G)|$. This bound is tight since
$K_4$ has no two independent vertices.

To find a maximum induced $1$-degenerate subgraph, we need a large induced
forest in $G$. Borodin \cite{borodin} proved that every planar graph $G$
is acyclically $5$-colourable, that is, we can partition the vertices of $G$
into five classes such that the subgraph induced by the union of any two classes is acyclic (hence, a
forest). By taking two largest classes we can guarantee an induced forest of
order at least $2/5 \cdot |V(G)|$ in $G$. The Albertson-Berman conjecture
\cite{AB} asserts that every planar graph has an induced forest containing at least
half of its vertices. This conjecture is tight as $K_4$ has no induced forest of
order greater that $2$. Borodin and Glebov \cite{BG} proved that the
Albertson-Berman conjecture is true for planar graphs of girth at least $5$.

Let $G$ be a plane graph. The vertices that belong to the unbounded face induce
an outerplanar graph. Let us delete them. The vertices
that belong to the unbounded face again induce an outerplanar graph and we can repeat
the process. In this way we create a sequence of outerplanar layers such
that only vertices in neighbouring layers can be adjacent.
If we take every second layer, the vertices from these layers induce a disjoint
union of outerplanar graphs. Since outerplanar graphs are $2$-degenerate
(every outerplanar graph contains a vertex of degree at most $2$, see
\cite{LW}), we can partition the vertices of $G$ into two sets such that each
set induces a $2$-degenerate graph. The larger of  these two sets has at
least $1/2 \cdot |V(G)|$ vertices. On the other hand, in the octahedron we can
take at most $4$ vertices into an induced $2$-degenerate subgraph, so the
maximum $2$-degenerate subgraph has order $2/3 \cdot |V(G)|$.

Degeneracy $3$ was studied by Oum and Zhu \cite{qsz} who were
interested in the order of a maximum $4$-choosable induced subgraph of a planar
graph. They showed that every planar graph has an induced $3$-degenerate
subgraph of order at least $5/7 \cdot |V(G)|$. For the upper bound,
the best we are aware of is that both octahedron and
icosahedron contain no induced $3$-degenerate subgraph of order greater than
$5/6 \cdot |V(G)|$.

To the authors' knowledge, there are no published results
concerning maximum $4$-degenerate induced subgraphs of planar
graphs. A likely reason is that such bounds are not interesting for
list-colouring applications: Thomassen proved that every planar graph is
$5$-choosable \cite{tomassen}.

The problem of  maximum degenerate subgraphs was also studied for general graphs
by Alon, Kahn, and Seymour \cite{aks}.
They precisely determined how large $k$-degenerate induced subgraph one can
guarantee depending only on the degree sequence of $G$.

This paper focuses on degeneracy $4$.
We define two operations for vertex removal: deletion and collection.
To \emph{delete} a vertex $v$, we remove $v$ and its incident edges from the
graph.
To \emph{collect} a vertex $v$ is the same as to delete $v$, but to be able to
collect $v$ we require $v$ to be of degree at most $4$.
Although the definitions are very similar, for our purpose there is a clear
difference: we want to collect as many vertices as possible and delete as few
as possible. The collected vertices induce a $4$-degenerate
subgraph whose order we are trying to maximize. We say we can \emph{collect a
set} $S$ of vertices if there exists a sequence in which the vertices of $S$ can
be collected. Vertices that are deleted or collected are collectively called
\emph{removed}. Note that a graph $G$ is $4$-degenerate if and only if we can
collect all its vertices.

\medskip

The main results of this paper are the following two theorems.
\begin{theorem}\label{mt}
Every connected planar graph with average degree $d \ge 2$   has
a $4$-degenerate induced subgraph containing at least $(38-d)/36$ of its vertices.
\end{theorem}
\begin{theorem}\label{smt}
In every planar graph with at least $7$ vertices we can delete a vertex in such a
way that we can collect at least $6$ vertices.
\end{theorem}

Since the average degree of a planar graph is less than $6$, Theorem~\ref{mt} has the following corollary.
\begin{corollary}\label{cmt}
In every planar graph $G$ we can delete less than $1/9$ of its vertices
in such a way that we can collect all the remaining ones.
The collected vertices induce a $4$-degenerate subgraph of $G$ containg more than $8/9$ of its vertices.
\end{corollary}

These results are probably not the best possible.
In the icosahedron, we need to delete one vertex out of twelve
to be able to collect the remaining eleven. We believe that this is the worst possible case.

\begin{conjecture}
In every planar graph $G$ we can delete at most $1/12$ of its vertices
in such a way that we can collect all the remaining ones.
\end{conjecture}
\begin{conjecture}
In every planar graph with at least $12$ vertices we can delete a vertex in such
a way that we can collect at least $11$ vertices.
\end{conjecture}

\section{Induction invariants}
\label{sec:invariants}

To prove Theorem \ref{smt} we only need to find a vertex whose deletion allows us
to collect $6$ vertices
in the neighbourhood. To prove Corollary~\ref{cmt} in this straightforward manner
we would need to collect $8$ vertices
per one deleted vertex. We cannot guarantee this immediately in all cases,
but even if we collect only $6$
vertices, we do something else that helps us: we create a large face and thus decrease
the average degree of the graph.
For a planar graph $G$, let  
\begin{eqnarray}
\Phi(G) &=& \sum_{v\in V(G)} (\deg(v) - 5), \label{definephi}\\
\Gamma(G) &=& {1\over 12}|V(G)| +
\frac 1{36}\Phi(G)  + \frac 1{18}\tc(G), \label{defineGamma}
\end{eqnarray}
where  $\tc(G)$ is the number of tree components of $G$.
 Theorem \ref{thm:main2} below    is the actual theorem we are going to prove.
\begin{theorem}\label{thm:main2}
Suppose that $G$ is a planar graph. The following is true:
\begin{itemize}
\item[(1)]
We can collect all vertices of $G$, or delete a vertex of $G$ and then
collect at least $6$ vertices.
\item[(2)]
There is a set $S\subset V(G)$ with at most $\Gamma(G)$ vertices such that if we delete $S$  then   we
can collect all the remaining vertices of $G$.
\end{itemize}
\end{theorem}

Theorem~\ref{mt} is implied by Theorem~\ref{thm:main2} (2):
if $d$ is the average degree of $G$, then $\Phi(G)=(d-5)|V(G)|$ and if $d\ge 2$, then $G$ is not a tree, thus
$$
\Gamma(G) = \frac{1}{12}|V(G)|+\frac{1}{36}(d-5)|V(G)|=\frac{d-2}{36}|V(G)|.
$$

\begin{lema}
Any smallest counterexample to Theorem \ref{thm:main2} is connected.
\label{lemma:connected}
\end{lema}

\begin{proof}
Let $G$ be a smallest counterexample to Theorem \ref{thm:main2} which is not connected. Let $G_1$ be
a component of $G$
and let $G_2=G-G_1$.

Suppose that statement (1) of Theorem \ref{thm:main2} does not hold for
$G$.
Then $G$ cannot be collected. Therefore either $G_1$ or $G_2$ cannot be
collected.
Since $G$ is the smallest counterexample to Theorem \ref{thm:main2}, the graph
$G_1$ or $G_2$ contains a vertex whose deleting allows us to collect $6$
vertices,
a contradiction.

Suppose that statement (2) of Theorem \ref{thm:main2} does not hold for
$G$ but it holds for $G_1$ and $G_2$. We obtain sets $S_1$ and $S_2$
satisfying the conditions of statement (2) of Theorem \ref{thm:main2}.
The union of $S_1$ and $S_2$ satisfies
statement (2) of Theorem \ref{thm:main2}, a contradiction.
\end{proof}

\begin{lema}
Any smallest counterexample to Theorem \ref{thm:main2} does not contain a vertex
that can be collected.
\label{lemma:mindegree}
\end{lema}
\begin{proof}
Assume a vertex $v$ of $d \le 4$ is collected. Then
\[
\Gamma(G-v) = \Gamma(G)- \frac 1{12} + \frac 1{36}(5-d) - \frac 1{36} d + \frac 1{18}(\tc(G-v)-\tc(G)).
\]
Here $- 1/12$ is due to the decrease of the number of vertices
by one,   $+\frac 1{36}(5-d)$ is due to the removal of $v$ from the
sum defining $\Phi$
(equation (\ref{definephi})),
 $-\frac 1{36} d$ is due to the fact that  neighbours of $v$ are of
smaller degree after deleting $v$,
 the last term is due to the change of the number of tree components.
 Since deleting $v$ increases the number of
tree components   by at most $d-1$,   we conclude that
$
\Gamma(G-v)   \le \Gamma(G)$.
By the minimality of $G$, there is a set $S \subset V(G-v)$ with at most $\Gamma(G-v) \le \Gamma(G)$ vertices
such that deleting $S$ allows us to collect the remaining vertices of $G$.
\end{proof}

In the remainder of this paper, let $G$ be a plane graph which is a counterexample to Theorem \ref{thm:main2} of smallest order.
We shall derive a contradiction via discharging. Thanks to Lemmas~\ref{lemma:connected} and~\ref{lemma:mindegree}, $G$ is connected and has minimum degree $5$.

\section{Discharging procedure}
\label{sec:procedure}

For a positive integer $k$, a vertex of degree $k$ is called a $k$-vertex, and a vertex of degree at least $k$ is a $k^+$-vertex.
Similarly, a $k$-neighbour is a neighbour of degree $k$, and a $k^+$-neighbour is a neighbour of degree at least $k$.
A face of $G$ of length $k$ is referred to as a $k$-face and a face of length at least $k$ is referred to as $k^+$-face.
The sets of $k$-vertices, $k^+$-vertices, and $k^+$-faces of $G$ are denoted by $V_k$, $V_{k^+}$, and $F_{k^+}$, respectively.
The sets of all vertices, edges, and faces of $G$ are denoted by $V$, $E$, and $F$, respectively.

In figures, we put $k$ or $k^+$ besides a $k$- or a $k^+$-vertex, respectively (see Figure~\ref{fdd} for an example).
If no number is given, then the vertex has degree $5$ or $6$.

Each vertex of $V_{6+}$ is assigned a certain {\it type\/} according to Table~\ref{figt}.
If $w$ is of degree $d$,   contained  in  at least $n_{\nontr}$ non-triangular faces,
and has at most $n_5$ neighbours from $V_5$, then $w$ can have type $t$.
The type of $w$ is the type that occurs first in the table among
all the types $w$ can have.
The symbol $n_5=3c$ means that $w$ has exactly three $5$-neighbours
and all of them are consecutive in the embedding of $G$.

Let $vw$ be an edge such that $v\in V_5$ and $w\in V_{6+}$.
For every such edge we define the  maximum charge $\mc(v,w)$ that $v$ can
send to $w$. This  maximum  charge is given in the last column of Table~\ref{figt}.
If $w$ is of type 9c or 8c, then $\mc(v,w)$ depends on the position of
$v$ with respect to $w$:  the value of $\mc(v,w)$ is $1$ when $v$ is the central one of the three consecutive
$5$-neighbours of $w$ and $9/10$ otherwise.

\begin{table}
\begin{center}
\begin{tabular}{ccccc}
Type & Degree & Min. number& Max. number & Max. \\
($t$)&($d$)  & of non-tr. faces & of $V_5$ neigh. & charge \\
& & ($n_{\nontr}$) & ($n_5$) & ($\mc$)\\
\hline
10a & 10+ & 0 & 3 & 1 \\
10b & 10+ & 0 & $\infty$ & 1/2 \\
\hline
9a & 9 & 1 & 3 & 1 \\
9b & 9 & 0 & 2 & 1 \\
9c & 9 & 0 & 3c& 9/10, 1, 9/10 \\
9d & 9 & 0 & 9 & 1/2 \\
\hline
8a & 8 & 0 & 1 & 1 \\
8b & 8 & 1 & 2 & 1 \\
8c & 8 & 2 & 3c& 9/10, 1, 9/10 \\
8d & 8 & 0 & 2 & 9/10 \\
8e & 8 & 0 & 8 & 1/2 \\
\hline
7a & 7 & 0 & 1 & 4/5 \\
7b & 7 & 1 & 2 & 13/20 \\
7c & 7 & 0 & 2 & 2/5 \\
7d & 7 & 0 & 7 & 1/3 \\
\hline
6a & 6 & 1 & 1 & 2/5 \\
6b & 6 & 0 & 6 & 0 \\
\hline
\end{tabular}
\end{center}
\caption{Maximal charges that can be send to a vertex.}
\label{figt}
\end{table}

First, we assign certain initial charges to the vertices and faces of $G$.
Each $d$-vertex receives charge $6-d$ and each $\ell$-face receives charge $2(3-\ell)$.
In the following discharging procedure, we redistribute the charges between
vertices and faces in a certain way such that no charge is created or lost.
The  initial and final  charge of a vertex or a face $x$ is denoted by  $ch_0(x)$ and $\ch(x)$,   respectively. For a set $S\subset
V\cup F$, the expression $\ch(S)$ denotes the {\it total charge} of the set $S$, that
is, the sum of charges of the elements of $S$.

 By  Euler's theorem, the initial total charge $\ch_0(V\cup F)$ is equal
to
$$
\sum_{v\in V} (6-\deg(v)) + \sum_{f\in F} 2(3-\ell) = 6|V|-2|E| + 6|F|-4|E| =
12.
$$
Our aim is to move charge from vertices to faces. Note that only the $5$-vertices have positive initial charge.
The discharging procedure consists of the following three steps.

\medskip

\noindent {\bf Step 1: Discharging to faces.}
For each vertex $v$ and for every face $f\in F_{4+}$ that contains $v$ do the following:
\begin{enumerate}
\item If $v$ is of degree $6$, then send $2/5$ from $v$ to $f$.
\item If $v$  is not of  degree $6$, but both its neighbours on $f$ have degree $6$,
then send $3/5$ from $v$ to $f$.
\item If $v$  is not of  degree $6$ and one of its neighbours  is not of  degree $6$,
then send $1/2$  from $v$ to $f$.
\end{enumerate}

\medskip

\noindent {\bf Step 2: Distance discharging.}
In every subgraph of $G$ isomorphic to the configuration in Figure \ref{fdd} send
$1/5$ from vertex $v$ to vertex $w$
(vertices are denoted as in Figure \ref{fdd}; the depicted vertices are pairwise distinct).
\begin{figure}[b]
\begin{center}
\includegraphics[scale=.8]{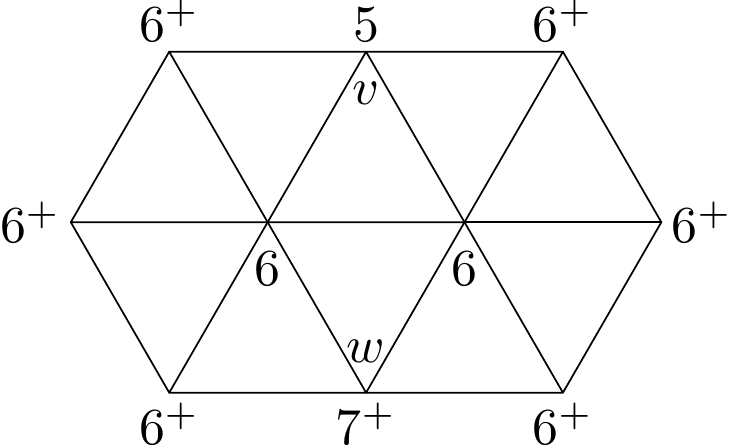}
\end{center}
\caption{Distance discharging.}
\label{fdd}
\end{figure}

\medskip

\noindent {\bf Step 3: Final discharging of the $5$-vertices.}
For each vertex $v\in V_5$ carry out the following procedure.
Order the neighbours $w\in V_{6+}$ of $v$ according to the value of $\mc(v,w)$ starting with the largest value;
let $w_1, w_2, \dots$ be the resulting ordering. If the value of $\mc(v,w)$ is the same for two neighbours
of $v$, then we order them arbitrarily.
For $i=1, 2, \dots$, send $ \min \{\mc(v,w_i) , \ch_a(v)\}$ from $v$ to $w_i$,
where $\ch_a(v)$ denotes the current charge of $v$.
If $\mc(v,w_i)  \ge  \ch_a(v)$, then we say that $v$ \emph{completely discharges} into $w_i$.

\medskip

The discharging procedure we have carried out has not changed the total charge.

\begin{lema} \label{lemma:f}
After the discharging procedure, no face has positive charge. Consequently,
$\ch(V)\ge 12$.
\end{lema}

\begin{proof}
The only step where a face can obtain positive charge is Step 1. Let $f$ be a
non-triangular face
of $G$.  Let $l$ be the length of $f$. Note that if a vertex sends $3/5$ to $f$,
then both neighbours of
$v$ on $f$ send $2/5$. Therefore the number of vertices that send $3/5$ to
$f$ is less than or equal to the number of vertices that send $2/5$ to $f$. This shows
that the face $f$ receives charge at most $l/2$, which is not enough to make $\ch(f)$ positive.
\end{proof}

Consequently, the sum of final charges of vertices of $G$ is positive.

\section{Avoiding small cut-sets}

A {\it triangle-cut\/} $C$ is a subgraph of $G$ isomorphic to $C_3$ such that
$V(C)$ is a cut-set of $G$.
A {\it chordless quadrilateral-cut\/} $C$ is an induced  subgraph of $G$ isomorphic to
$C_4$ such that $V(C)$ is a cut-set of $G$.
A {\it bad cut-set\/} is a triangle-cut, {or} a chordless quadrilateral-cut.

A \emph{good subgraph} of $G$ is either $G$ itself if $G$ contains no bad cut-set or a proper subgraph $H$ which
satisfies all the following conditions:
\begin{enumerate}
\item $H$ contains a bad cut-set $C$ of $G$.
\item  There is an embedding of $G$ such that
all the  vertices
of $G-V(H)$ are in the exterior of $C$, and all the vertices of $H -C$ are in the interior of $C$. If this condition is satisfied, then \emph{$C$ cuts $H$ from $G$}.
\item   $H-C$ contains no vertex $v$ that is in a bad cut-set of $G$.
\end{enumerate}

\begin{lema}\label{la}
The graph $G$ has a good subgraph $H$.
\end{lema}
\begin{proof}
Assume that $G$ has a triangle-cut $C_A$.
Then $G$ can be embedded in the plane in such a way that the interior and the exterior of $C$ are both nonempty.
We choose the triangle-cut $C_A$ and the embedding so that the interior contains the minimum number
of vertices.
Let $H_A$ be the subgraph of $G$ induced by $V(C_A)$ and the vertices in the interior of $C_A$.
Thus $C_A$ cuts $H_A$ from $G$.

If $H_A-V(C)$ contains a vertex that is in a triangle-cut $C'$ of $G$,
then the vertices of $C_A-C'$ all belong to the same component $K$ of
$H_A-V(C')$. Hence $G$ can be embedded in such a way that the interior of $C'$ contains
less vertices than that of $C_A$, contrary to our choice of $C_A$.

  If $G$ has no   triangle-cut, then let $C_A = \emptyset$ and $H_A=G$.

Assume $H_A-V(C_A)$ contains a vertex that is in a chordless
quadrilateral-cut $C$ of $G$.
We choose $C$ and an embedding of $G$ in such a way that the component of $G-C$ containing $C_A$ lies
in the exterior of
$C$ (if $V(C_A)=\emptyset$, then we choose an arbitrary component), and the number of vertices that lie in the
interior of $C$ is minimum but nonzero. Let $H$ be the subgraph of $G$ induced by $V(C)$ and the vertices in the interior of $C$.
Thus $C$ cuts $H$ from $G$.

Assume $C=v_1v_2v_3v_4v_1$.
If $H-V(C)$ contains a vertex $v$ that is in a chordless
quadrilateral-cut $C'$ of $G$,  as only
 non-adjacent vertices can belong to different
components of $G-C'$,   we may assume
that $v_1$ and $v_3$ are in different components of $G-C'$. Consequently,
vertices $v_2$ and $v_4$ must be in $C'$.
As vertex $v$ has degree at least $5$, $v$   has a neighbour $x$ distinct from $v_1,v_2,v_3,v_4$.
The vertex $x$ lies either inside $vv_2v_1v_4$ or $vv_2v_3v_4$.
Assume $x$ lies inside   $vv_2v_1v_4$.
Let $C'' = vv_2v_1v_4$. The circuit $C''$ is a chordless quadrilateral-cut of $G$:
the vertex $x$ is inside $C'$, vertex $v_3$ is outside $C'$,
$v_2v_4$ is not a chord since $C$ has no chord, and finally
$v_1v$ is not chord because if it was, then either
$v_1vv_2$ or $v_1vv_4$ is a triangle-cut of $G$. The interior of $C''$ has fewer vertices than that of $C$, contradicting our choice of $C$.

If  $H_A-V(C_A)$ contains no vertex that is in a chordless
quadrilateral-cut  of $G$. then $H=H_A$ and $C=C_A$. In any case, $H$ and $C$ satisfy the conditions in the definition of a good subgraph.
\end{proof}

Let $H$ be a good subgraph whose existence is guaranteed by Lemma~\ref{la} and $C$ the cut separating it from the rest of $G$ ($C$ is empty if $H=G$).
The vertices from $V(C)$ will be called \emph{cut vertices}.
In case $C$ is a triangle-cut, vertices in $H-V(C)$ that are adjacent to two vertices of $C$ are
\emph{extraordinary}. In all possible cases for $C$, vertices from $H-V(C)$  which are not extraordinary are called \emph{ordinary}.

\begin{lema}
The graph $H$ contains an extraordinary vertex $v$ with $\ch(v) \ge 2$ or the sum of the final charges of ordinary vertices in $H$ is positive.
\label{lemanoext}
\end{lema}

\begin{proof}
If $V(C) = \emptyset$, then $H=G$ and all vertices are ordinary; the statement of the lemma is immediately implied by Lemma
\ref{lemma:f}.
We split the rest of the proof into two cases according to the type of the bad cut-set $C$---either
$C$ is a triangle-cut or $C$ is a chordless quadrilateral-cut.
In our embedding of $G$, the inner vertices and the inner faces of $C$ will be collectively called the \emph{kernel}
of $H$ and will be denoted by $\K$. The ``outer face'' of $H$ (bounded by $C$) is denoted by $f$.

\smallskip

\noindent {\bf Case 1: $C$ is a chordless quadrilateral-cut.}
There are no extraordinary vertices in this case.
Let $C=v_1v_2v_3v_4v_1$.
Let us compute the initial charge on the kernel of $H$.
Note that this charge is the same as if we have assigned initial charges in $H$ instead of $G$.
When the initial charges are assigned to vertices and faces of $H$ as a plane graph, the total charge is $12$.
Since the outer face has charge $-2$ and the vertex $v_i$ of $C$ has charge $6-\deg_H(v_i)$,
the initial charge of the kernel $\K$ is
$$
\ch_0(\K) = 12+2-\sum_{i=1}^4 \left( 6-\deg_H(v_i)\right).
$$

The initial charge of $\K$ may get lost via discharging to $V(C)$ at Step 3 or via distance discharging at Step 2.
Each $5$-vertex $v$ in $\K$ has initial charge $1$. Such a vertex $v$ first sends charge at least $2/5$
to each of its incident non-triangular faces, then sends charge $1/5$ to a distance $2$ neighbour if Step 2 applies, and then
sends some of the remaining charge to its $6^+$-neighbours.

For simplicity, we first assume that no charge of $\K$ is lost via distance discharging.
Let $f_i$ be the bounded face of $H$ incident to $v_iv_{i+1}$ for $i\in \{1,2,3,4\}$ (indices taken modulo $4$, i.e., $f_4$ is incident to $v_4v_1$).
If $f_i$ is a triangle, then $v_i$, $v_{i+1}$ have a common neighbour. Assume $q$ of the  faces $f_1,f_2,f_3,f_4$ are triangles, and $4-q$
of them are $4^+$-faces. Then the number of $5$-vertices in $\K$ adjacent to $V(C)$ is at most $\sum_{i=1}^4 (\deg_H(v_i)-2) -q$,
and the amount of charge send from these vertices to faces in $\K$ is at least $(4-q)\cdot 4/5$. Hence the total amount of charge sent from $\K$ to
$V(C)$ is at most
  \[
  \sum_{i=1}^4 (\deg_H(v_i)-2) -q - (4-q)\frac 45 \le \sum_{i=1}^4 (\deg_H(v_i)-2) - \frac {16}5
  \]
Therefore, the final charge of $\K$ is
\begin{eqnarray}
\ch(\K) \ge 12+2-\sum_{i=1}^4 \left( 6-\deg_H(v_i)\right) -
\sum_{i=1}^4 \left( \deg_H(v_i)-2 \right) + \frac{16}{5} > 0. \label{eqn}
\end{eqnarray}

\smallskip

If some vertices   send charge  from $\K$ via distance discharging out of $\K$, then
each such instance of distance discharging implies that two neighbours of $V(C)$ in $\K$ have degree $6$
which prevents these two vertices to discharge any charge outside
$\K$. So the total charge send from $\K$ to $G-\K$ is less than the amount estimated above.

Thus in any case, $\ch(\K) > 0$. Thanks to Lemma \ref{lemma:f}, faces do not have
positive charge and so the sum of charges of ordinary vertices is positive.
\bigskip

\noindent {\bf Case 2: $C$ is a triangle-cut.}
Let $v_1$, $v_2$, and $v_3$ be the vertices of $C$.
 Similarly as in Case 1, the initial charge of the kernel is
$$
\ch_0(\K) = 12-\sum_{i=1}^3 \left( 6-\deg_H(v_i)\right).
$$
Let $X$ be the set of vertices in $\K$ adjacent to $V(C)$.
Each $v \in X$ may send charge $1$ out of $\K$.
Thus the total charge send from $\K$ to $G-\K$ is bounded above by
$$|X| \le \sum_{i=1}^3 \left( \deg_H(v_i)-2 \right).$$
If $\K$ has no extraordinary vertex, then the bounded faces of $H$ incident to the edges of  $C$ are $4^+$-faces.
Hence vertices in $X$ send at least $12/5$ of the charge to faces in $\K$. So the
 total charge send from $\K$ to $G-\K$ is bounded above by  $|X|-12/5$, implying that
\[
\ch(\K) \ge 12-\sum_{i=1}^3 \left( 6-\deg_H(v_i)\right) - \sum_{i=1}^3 \left( \deg_H(v_i)-2 \right)   +   \frac{12}5 > 0.
\]

Assume $\K$ has extraordinary  vertices. Let $Y_5$ be the set of extraordinary $5$-vertices
and $Y_{6^+}$ be the set of extraordinary $6^+$-vertices.
As each extraordinary  vertex is adjacent to two vertices of $V(C)$, we have $$|X| \le \sum_{i=1}^3 \left( \deg_H(v_i)-2 \right)-(|Y_5|+|Y_{6^+}|).$$
Assume that the vertices from $Y_5$ sent total charge $c$ out of $\K$. Vertices from $Y_{6^+}$ do not send any charge out of $\K$. Finally, ordinary vertices can send at most $|X|-(|Y_5|+|Y_{6^+}|)$, and so the total charge send from $X$ to
$G-\K$ is at most $c+|X|-(|Y_5|+|Y_{6^+}|)$. Therefore,
\begin{eqnarray*}
\ch(\K) &\ge& \ch_0(\K) - \sum_{i=1}^3 \left( \deg_H(v_i)-2 \right) + 2(|Y_5|+|Y_{6^+}|)-c\\
&=& 2(|Y_5|+|Y_{6^+}|)-c.
\end{eqnarray*}
Note that $\ch(Y_5) \le |Y_5|-c\le 2|Y_5|-c$ (the last inequality is strict if $Y_5\neq\emptyset$).
If a vertex of $Y_{6^+}$ has final charge at least $2$, the lemma is proved; otherwise $\ch(Y_{6^+}) \le 2|Y_{6^+}|$ (and the inequality is strict if $Y_{6^+} \ne \emptyset$). Since $Y_5 \cup Y_{6^+}\neq \emptyset$, we have $2|Y_5|-c +2|Y_{6^+}|>\ch(Y_5 \cup Y_{6^+})$.
Therefore $\ch(\K)>\ch(Y_5 \cup Y_{6^+})$ and so the final charge of $\K-(Y_5 \cup Y_{6^+})$  is positive.

 Thanks to Lemma \ref{lemma:f}, faces do not have
positive charge and so the sum of charges of ordinary vertices is positive.
\end{proof}

\section{Analysis of configurations}

According to Lemma \ref{lemanoext}, after discharging, $H$ has an ordinary positive vertex or
an extraordinary vertex with charge at least $2$. Our aim is to prove that both
cases lead to a contradiction. In either case, we will do it by examining an
exhaustive list of configurations and showing that none of those configurations
can occur in the minimal counterexample $G$.

In each of the configurations, we obtain a graph $G'$ from $G$ by
deleting a vertex $v$ and collecting several
other vertices (always at least six) subsequently.
This ensures that statement (1) of
Theorem~\ref{thm:main2} is true for $G$. The graph $G'$ has less vertices than $G$,
hence Theorem \ref{thm:main2} holds for $G'$ and thus there exists a subset $S'\subseteq V(G')$
 of order at most $\Gamma(G')$  whose deletion allows us to collect the remaining vertices of $G'$.
Consequently, the deletion of the set $S=S' \cup \{v\}$ allows us to collect all
the vertices of $G$. If $\Gamma(G) \ge \Gamma(G')+1$, then this contradicts the
fact that $G$ is a counterexample
to Theorem \ref{thm:main2}. This would show that the examined configuration is
not contained in $G$.

The critical part is to prove $\Gamma(G) \ge \Gamma(G')+1$. The
computation consists of several steps; we demonstrate it in full detail in the
proof of Lemma~\ref{lemaexample}. After that, we introduce a short notation that
will allow us to skip repetitive arguments and help the reader to track all the
details.

\begin{lema}\label{lemaexample}
The graph $G$ has no extraordinary vertex $v$ with $\ch(v)\ge 2$.
\end{lema}
\begin{proof}
Suppose, for a contradiction, that $G$ contains such a vertex $v$.
Any $5$-vertex has initial charge $1$ and then its charge only decreases, hence $v\in V_{6+}$.
The initial charge of $v$ is $6-\deg_G(v)$, so it
has to receive charge at least $\deg(v)-4$ during the discharging phase.
Note that $v$ can receive at most $\deg(v)/5$ via distance discharging, so most of the charge
has to come from $5$-neighbours in Step~3.
Moreover, any distance discharging reduces the possible number of $5$-neighbours.
A short case analysis left to the reader shows that $v$ must be of type $8e$.
In addition, $v$ must be surrounded by $8$ vertices $v_1,\dots, v_8$ of degree $5$,
and all the faces surrounding $v$ are triangles.

We can delete $v$ and collect $v_1, \dots , v_8$, obtaining a graph $G'$.
All we have to show is that $\Gamma(G) \ge \Gamma(G') +1$.
The function $\Gamma$ depends on three parameters: the number of vertices, the
value of the function $\Phi$,
and the number of tree components.  The  removal of vertices $v, v_1,
\dots , v_8$ affects the value of $\Gamma$ as follows:
\begin{itemize}
\item[1.]
The number of vertices decreases by $9$.

\item[2a.]
The value of $\Phi$ decreases because the removed vertices do not contribute to the sum \eqref{definephi} anymore.
In most of our configurations, we do not know the degree of every vertex precisely,
but for every vertex $w$ we have a lower bound $\text{mindeg}(w)$ on its degree.
Let $b_i$ be the number of vertices with $\text{mindeg}(w)=i$ for $i\in \{5, \dots
,9\}$ and let $b_{10}$ be the number of vertices with $\text{mindeg}(w)\ge 10$.
Let $\text{\bf b}=(b_5, b_6, b_7, b_8, b_9, b_{10})$; in our case,
$\text{\bf b}=(8,0,0,1,0,0)$. The value of $\Phi$ is decreased by at least
$b_6+2b_7+3b_8+4b_9+5b_{10}$ because of the removal of the vertices from the sum
(\ref{definephi}). In our case, $\Phi$ decreased by at least $3$.

\item[2b.] The value of $\Phi$ is also decreased because the neighbours of the removed
vertices have smaller degree in $G'$.
Let $\Sigma_d$ denote the sum of the degrees of the
removed vertices; clearly $\Sigma_d \ge 5b_5+6b_6+7b_7+8b_8+9b_9+10b_{10}$.
Let $\Sigma_e$ be the number of edges
of the graph induced by $v, v_1, \dots , v_8$. Clearly, the value of
$\Phi$ decreases by at least
$\Sigma_d-2\Sigma_e$ because of the neighbours' degree reduction.

The only remaining problem is to count $\Sigma_e$. We have $8$ edges between $v$
and $v_1, \dots , v_8$.
We have another $8$ edges between vertices $v_1, \dots , v_8$ on the triangular
faces containing $v$.
No other edge may exist in the subgraph of $G$ induced by $v, v_1, \dots , v_8$. Such an
edge would together with $v$ induce a triangle-cut of $G$, which contradicts that $v$ is extraordinary.
Therefore $\Sigma_e\le16$. (Note that we only need an upper bound on $\Sigma_e$.)
We conclude that $\Phi$ decreases by  at least  $16$ because of decreased degrees of the
neighbours of removed vertices.
In total $\Phi$ decreases by  at least $19$.

\item[3.]  Now we  show that no new tree component is created (creation of tree components during reduction increases $\Gamma$).
Assume that such a tree component exists; it must have a vertex $w$ of degree
at most $1$ in $G'$. The vertex $w$ therefore has at least four neighbours
among the deleted vertices. All four
these neighbours are neighbours of $v$;
let $v_i$ and $v_j$ be two non-adjacent of them (note that we only need three of the four neighbours to proceed with this kind of argument).
The cycle $vv_iwv_jv$ is a chordless quadrilateral cut, and this contradicts the assumption that $v$ is extraordinary.
Therefore, no tree component is created.

This argument also holds in configurations when $w$
can be a neighbour of $v$. In this case a triangle-cut containing $v$, $w$ and some other neighbour of $v$
is created.
\end{itemize}
The value of $\Gamma$ decreases by at least $9/12+19/36 =23/18$.  
\end{proof}

The reduction idea used in the proof of Lemma \ref{lemaexample} will be used many times.
We summarize the key points of the calculation in the following lemma; its statement is structured with respect to the phase of the computation.
\begin{lema}\label{lemacomp}
Suppose that we  remove an ordinary vertex $v$ together with some neighbours $v_1, \dots, v_{m_1}$
of $v$ and some vertices $v_1', \dots, v_{m_2}'$ at distance $2$ from $v$.
The following statements are true:
\begin{itemize}
\item[($\Sigma_e$):] Two neighbours $v_i$ and $v_j$ of $v$ are joined by an edge if and only if $vv_iv_j$ is a triangular
face. The vertex $v'_k$, for $k\in\{1,\dots,m_2\}$, can be adjacent to at most two neighbours of $v$. If $v'_k$ is adjacent to two
neighbours of $v$, say $v_i$ and $v_j$, then either $v'_kv_ivv_j$ is a $4$-face, or $v'_kv_iv_j$ and $vv_iv_j$ are triangular faces.
\item[($\Delta\Phi$):] $\Delta\Phi \ge 5b_5+7b_6+9b_7+11b_8+13b_9+15b_{10}-2\Sigma_e$.
\item[($\tc$):] If $m_2\le 1$ (that is, we removed at most one vertex at distance $2$ from $v$), then no new tree component is created.
\item[($\Delta \Gamma$):] $  \Delta \Gamma = \Delta |V|/12 +   \Delta \Phi /36 +  \Delta\tc/18$.
\end{itemize}
\end{lema}
\begin{proof}
Statement ($\Sigma_e$) is implied by the absence of bad cuts containing $v$; if $v_k'$ is adjacent to both $v_i$ and $v_j$ and $v_iv_j\in E(G)$, then $v_iv_jv_k'$ is a triangular face (for otherwise the edge $v_iv_j$ would be contained in a triangle-cut in $G$ and $v$ would be extraordinary).
Statement ($\Delta\Phi$) is implied by the definitions of $\Phi$ and $\text{\bf b}$.
Statement ($\tc$) is implied by the argument from the proof of Lemma \ref{lemaexample} (enumeration of the change of $\Gamma$, part 3).
Statement ($\Delta \Gamma$) follows from the definition of $\Gamma$.
\end{proof}

To make our computations easier to follow we will present the key
points in a concise form (the following example captures the computation used in proving Lemma~\ref{lemaexample}):
\begin{eqnarray*}
&\text{Delete}(v) \ \ \ \text{Collect}(v_1, v_2, v_3, v_4, v_5, v_6, v_7, v_8)
&\\
&\Delta |V|=9 \ \ \ \text{\bf b}=(8,0,0,1,0,0) \ \ \ \Sigma_e\le16 \ \ \ \Delta
\Phi\ge 19 \ \ \ \Delta\tc=0 \ \ \ \Delta \Gamma \ge 23/18.&
\end{eqnarray*}
Whenever possible, the order of the collected vertices will match the order in which
these vertices can be collected.

All the computations except for the computation of $\Sigma_e$ are entirely routine.
If Lemma \ref{lemacomp} cannot be used or a further clarification is needed, we signalize it by a star. The explanation will follow in square brackets.

\section{Charges on vertices of degree $6$ and more}

This section shows that no ordinary $6^+$-vertex may have
positive charge. Since the initial charge of any such vertex is not positive, its positive charge can only be obtained
during the discharging phase. First, we determine the  maximum charge a
vertex can obtain via distance discharging in Step 2 of the discharging procedure described in Section~\ref{sec:procedure}.

\begin{lema}\label{l7}
If a vertex $w$ of degree $k \ge 7$ has at least $m$ neighbours
of degree $5$, then $w$ can receive charge at most
$\lfloor (k-m-1)/3 \rfloor /5$ for $m>0$
and at most  {$\lfloor  k/3\rfloor/5$} for $m=0$ via distance discharging.
\end{lema}
\begin{proof}
If a vertex $w$ of degree $k \ge 7$ receives charge $1/5$ via distance
discharging, it must have four consecutive neighbours forming a path  {$(v_1,v_2,v_3,v_4)$ of degrees $6^+,6,6,6^+$, respectively.}
{Moreover, $v_2$ and $v_3$ have a common $5$-neighbour,  the vertex which sends charge $1/5$ to $w$.}

Note that $v_2$ and $v_3$ cannot be contained in any other path of this type:
they have only one $5$-neighbour.
The vertices $v_1$ and $v_4$ can be shared by two paths of this type around   {$w$}.
Consequently, {$w$} must have at least $3i$ neighbours from $V_{6+}$ to get
charge $i/5$ via distance discharging. For $m > 0$,   {$w$} must have at least $3i+1$ such neighbours.
This last two statements imply the lemma.
\end{proof}

\begin{corollary}\label{easytypes}
Vertices of types $10a$, $9a$, $9b$, $9c$, $8a$, $8b$, $8c$, $8d$, $7a$, $7b$,
$7c$, $6a$ and $6b$
cannot have positive charge.
\end{corollary}

\begin{proof}
We sketch the proof for a vertex $w$ of type 7a; the other cases are proved very similarly
(we need to employ $n_\nontr$ and look also at Step~1 in some of them).
The initial charge of $w$ is $6-7 = -1$.

If $w$ has a $5$-neighbour, it receives charge at most $4/5$
in Step~3 of the discharging procedure because it has only one $5$-neighbour
and that neighbour can send at most $\mc(w)$ {to $w$}.   {By} Lemma~\ref{l7}, it can receive
at most $\lfloor 5/3\rfloor/5 = 1/5$ in Step~2. Altogether,
  {the total amount of charge} $w$ receives is at most $4/5 + 1/5 = 1$, and so its final charge is not positive.

If $w$ has no $5$-neighbour, it receives nothing in Step~3
and at most $\lfloor 7/3\rfloor/5 = 2/5$ in Step~2 according to Lemma~\ref{l7}.
Thus $w$ has negative final charge.
\end{proof}

\begin{lema}\label{lema8}
The graph $G$ has no ordinary $8^+$-vertex with positive charge.
\end{lema}
\begin{proof}
 {Assume to the contrary that $v$ is an ordinary $8^+$-vertex with positive final charge. By Corollary \ref{easytypes},}   $v$ {is}  of type $10b$, $9d$ or $8e$.
If $v$ has at least six $5$-neighbours, {say $v_1, \dots , v_6$,}   {then we delete $v$ and collect $v_1, \dots , v_6$.}
For the purpose of counting $\text{\bf b}$, the lower bound $\text{mindeg}(v)$
on the degree of $v$ will be $8$.
\begin{eqnarray*}
&\text{Delete}(v) \ \ \ \text{Collect}(v_1, v_2, v_3, v_4, v_5, v_6) &\\
&\Delta |V|=7 \ \ \ \text{\bf b}=(6,0,0,1,0,0) \ \ \ \Sigma_e\le 11 \ \ \ \Delta
\Phi\ge 19 \ \ \ \Delta\tc=0 \ \ \ \Delta \Gamma \ge 10/9.&
\end{eqnarray*}

We prove that if $v$ is of type $10b$ or $9d$ {and $\ch(v) >0$}, then it must have at least six $5$-neighbours.
Indeed, if $v$ has exactly five $5$-neighbours, then it received charge at most $2.5$ from them.
Additional charge might have been received by distance discharging, but its amount is bounded from above by
$\lfloor (\deg(v)-5-1)/3 \rfloor /5$ due to Lemma~\ref{l7}. The initial charge of $v$ is $6-\deg(v)$, so its final charge
$6-\deg(v) + 2.5 + \lfloor (\deg(v)-5-1)/3 \rfloor /5$ is negative.
If $v$ has at most four $5$-neighbours, then $v$ receives charge at most $2$ from   {these} neighbours and
at most $\deg(v)/15$ by distance discharging (we used Lemma~\ref{l7} again).
The final charge $6-\deg(v) + 2 + \deg(v)/15$ is therefore negative.

The last two paragraphs show that $v$ is of type $8e$ and has at most five $5$-neighbours.
Since $\ch(v) > 0$, it is straightforward to verify that $v$ has at least four $5$-neighbours.

  {If} $v$   {has exactly}  four $5$-neighbours, {say $v_1, \dots, v_4$},
{then $v$ must receive charge $1/5$ from some vertex $u$ via distance charging. Hence
$v$ has four consecutive $6^+$-neighbours, including two $6$-vertices $v_5$, $v_6$ adjacent to $u$.
Moreover, all the faces incident to $v$ are triangular faces.
So $v_1, \ldots, v_4$ are consecutive neighbours of $v$ that form a path.}
Let $w$ be  {the common neighbour of $v$, $v_1$, and $v_5$.}
\begin{eqnarray*}
&\text{Delete}(w) \ \ \ \text{Collect}(v_1, v_2, v_3, v_4, v, v_5, v_6) &\\
&\Delta |V|=8 \ \ \ \text{\bf b}=(4,3,0,1,0,0) \ \ \ \Sigma_e\le 13 \ \ \ \Delta
\Phi\ge 26 \ \ \ \Delta\tc=0 \ \ \ \Delta \Gamma \ge 25/18.&
\end{eqnarray*}

 {Assume} $v$   has exactly five $5$-neighbours, {say $v_1, \dots v_5$}.
Only one face around $v$ can be non-triangular and therefore
we can always find a neighbour $w \not \in \{v_1, v_2, \dots, v_5 \}$ of $v$ that allows us to collect three
vertices from $\{v_1, v_2, \dots, v_5 \}$, which eventually allows us to collect $v$
and the remaining vertices from $\{v_1, v_2, \dots, v_5 \}$.
\begin{eqnarray*}
&\text{Delete}(w) \ \ \ \text{Collect}(v_1, v_2, v_3, v_4, v_5, v) &\\
&\Delta |V|=7 \ \ \ \text{\bf b}=(6,0,0,1,0,0) \ \ \ \Sigma_e\le 11 \ \ \ \Delta
\Phi\ge 19 \ \ \ \Delta\tc=0 \ \ \ \Delta\Gamma \ge 10/9.&
\end{eqnarray*}
\end{proof}


\begin{lema}\label{lema6}
The graph $G$ has no ordinary $6^+$-vertex with positive charge.
\end{lema}
\begin{proof}
For a contradiction, assume that $v$ is such a vertex.
The vertex $v$ must be of type $7d$ and have at least three $5$-neighbours due to Lemma~\ref{lema8} and Corollary~\ref{easytypes}.
Let $m\ge 3$ be the number of $5$-neighbours of $v$.

\medskip
\noindent{\bf Case $m=3$:}
The vertex $v$ must be only on triangular faces and receives charge by distance
discharging from some vertex $x$. {So the three $5$-neighbours of $v$, say $v_1, v_2, v_3$, are consecutive and form a path.}
Let $v_4$, $v_5$ be two $6$-vertices that are common neighbours of $v$
and $x$.
Let $w$ be   {the common neighbour of $v, v_1$ and $v_4$.}
\begin{eqnarray*}
&\text{Delete}(w) \ \ \ \text{Collect}(v_1, v_2, v_3, v, v_4, v_5) &\\
&\Delta |V|=7 \ \ \ \text{\bf b}=(3,3,1,0,0,0) \ \ \ \Sigma_e\le 11 \ \ \ \Delta
\Phi\ge 21 \ \ \ \Delta\tc=0 \ \ \ \DG \ge 7/6.&
\end{eqnarray*}

\medskip
\noindent{\bf Case $m=4$:}
Vertex $v$ must be only on triangular faces (note that distance
discharging into $v$
is impossible). We can easily find all four possible configurations around $v$;
they are displayed in Figure \ref{fig:7}. The neighbours of $v$
are denoted according to the figure.
\begin{figure}
\centering
\includegraphics{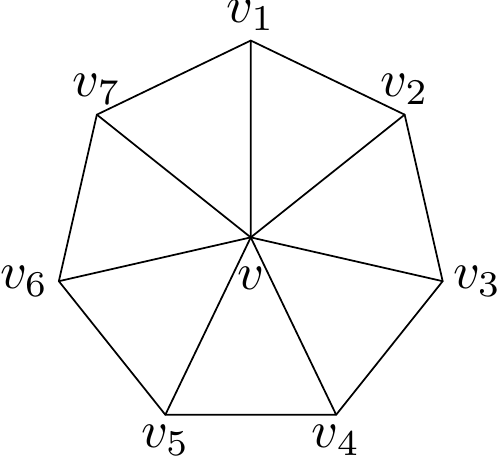}

\begin{tabular}{cccc}
\includegraphics{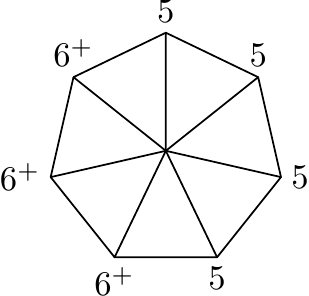} &
\includegraphics{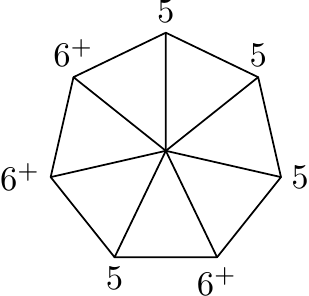} &
\includegraphics{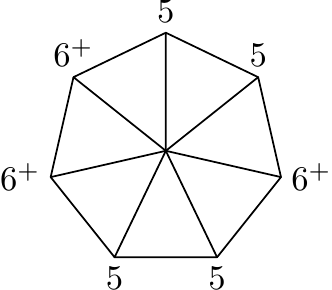} &
\includegraphics{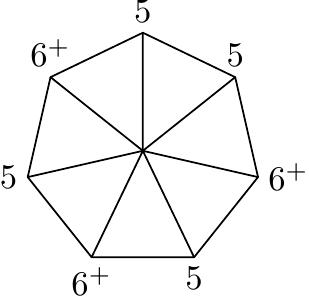}  \\
Conf. $1$ & Conf. $2$ & Conf. $3$ & Conf. $4$\\
\end{tabular}
\caption{Possible configurations around a $7$-vertex.}
\label{fig:7}
\end{figure}

\medskip\noindent\underline{\emph{Configuration 1:}}
If $v_7$ is of degree $6$, then
\begin{eqnarray*}
&\text{Delete}(v_5) \ \ \ \text{Collect}(v_4, v_3, v_2, v_1, v, v_7) &\\
&\Delta |V|=7 \ \ \ \text{\bf b}=(4,2,1,0,0,0) \ \ \ \Sigma_e\le 11 \ \ \ \Delta
\Phi\ge 21 \ \ \ \Delta\tc=0 \ \ \ \DG \ge 7/6.&
\end{eqnarray*}
If $v_7$ has another $5$-neighbour $w\neq v_1$, then
\begin{eqnarray*}
&\text{Delete}(v_7) \ \ \ \text{Collect}(w, v_1, v_2, v_3, v_4, v) &\\
&\Delta |V|=7 \ \ \ \text{\bf b}=(5,0,2,0,0,0) \ \ \ \Sigma_e\le 11 \ \ \ \Delta
\Phi\ge 21 \ \ \ \Delta\tc=0 \ \ \  \DG \ge 7/6.&
\end{eqnarray*}

{Assume $v_7$ has degree at least $7$ and has no other $5$-neighbour.}   {Then}  $\mc(v_1,v_7)\ge 4/5$.
If the edge $v_1v_7$ is on a non-triangular face $f$, the vertex $v_1$
would discharge at least $1/2$ into $f$ and then discharge completely into
$v_7$, leaving
nothing to discharge into $v$, {implying that $\ch(v) \le 0$}.  Therefore, there exist a common neighbour $w$ of
$v_1$ and $v_7$.
\begin{eqnarray*}
&\text{Delete}(w) \ \ \ \text{Collect}(v_1, v_2, v_3, v_4, v, v_7) &\\
&\Delta |V|=7 \ \ \ \text{\bf b}=(5,0,2,0,0,0) \ \ \ \Sigma_e\le 11 \ \ \ \Delta
\Phi\ge 21 \ \ \ \Delta\tc=0 \ \ \  \DG \ge 7/6.&
\end{eqnarray*}
Assume $v_7$ has degree at least $8$. Then $v_1$ completely discharges into $v_7$, leaving
nothing for $v_1$ to discharge into $v$. Consequently, $\ch(v) \le 0$, a contradiction.

\medskip\noindent\underline{\emph{Configuration 2:}}
The argument is very similar to Configuration $1$. It suffices to switch the roles
of $v_4$ and $v_5$.

\medskip\noindent\underline{\emph{Configuration 3:}}
  {If $v_7$ has degree at most $7$, then}
\begin{eqnarray*}
&\text{Delete}(v_6) \ \ \ \text{Collect}(v_5, v_4, v, v_1, v_2, v_7) &\\
&\Delta |V|=7 \ \ \ \text{\bf b}=(4,2,1,0,0,0) \ \ \ \Sigma_e\le 11 \ \ \ \Delta
\Phi\ge 21 \ \ \ \Delta\tc=0 \ \ \  \DG \ge 7/6.&
\end{eqnarray*}
{Assume $v_7$ has degree at least $8$.} If $v_7$ has another $5$-neighbour $w$ besides $v_1$, then
\begin{eqnarray*}
&\text{Delete}(v_7) \ \ \ \text{Collect}(v_1, v_2, v, v_4, v_5, w) &\\
&\Delta |V|=7 \ \ \ \text{\bf b}=(5,0,1,1,0,0) \ \ \ \Sigma_e\le 11 \ \ \ \Delta
\Phi\ge 23 \ \ \ \Delta\tc=0 \ \ \  \DG \ge 11/9.&
\end{eqnarray*}
  {If $v_7$ has no other $5$-neighbour,}  then $v_1$ discharges completely into $v_7$,
and hence $v$ does not have positive final charge in this configuration.

\medskip\noindent\underline{\emph{Configuration 4:}}
 {If $v_7$ has degree at most $7$, then}
\begin{eqnarray*}
&\text{Delete}(v_5) \ \ \ \text{Collect}(v_6, v_4, v, v_1, v_2, v_7) &\\
&\Delta |V|=7 \ \ \ \text{\bf b}=(4,2,1,0,0,0) \ \ \ \Sigma_e\le 11 \ \ \ \Delta
\Phi\ge 21 \ \ \ \Delta\tc=0 \ \ \  \DG \ge 7/6.&
\end{eqnarray*}
{Assume $v_7$ has degree at least $8$.} {If $v_5$ has degree at most $7$, then}
\begin{eqnarray*}
&\text{Delete}(v_7) \ \ \ \text{Collect}(v_1, v_2, v, v_4, v_6, v_5) &\\
&\Delta |V|=7 \ \ \ \text{\bf b}=(4,2,1,0,0,0) \ \ \ \Sigma_e\le 11 \ \ \ \Delta
\Phi\ge 21 \ \ \ \Delta\tc=0 \ \ \ \DG \ge 7/6.&
\end{eqnarray*}
Thus both $v_7$ and $v_5$ have degree at least $8$.
This implies that all the charge of $v_6$ is sent to $v_7$ and $v_5$, and $v_6$ discharges $0$ to $v$. Therefore, $v$  does
not have positive final charge in this configuration.

\medskip
\noindent{\bf Case $m=5$:}
There can be only one non-triangular face around $v$.
Simple case analysis shows that $v$ always has three consecutive neighbours,
say $x$, $y$ and $z$, such that
$x$ is not of degree $5$,
$y$ and $z$ are of degree $5$,
and edges $xy$ and $yz$ exist.
The other $5$-neighbours of $v$ will be denoted by $v_1, v_2, v_3$.
\begin{eqnarray*}
&\text{Delete}(x) \ \ \ \text{Collect}(y, z, v, v_1, v_2, v_3) &\\
&\Delta |V|=7 \ \ \ \text{\bf b}=(5,1,1,0,0,0) \ \ \ \Sigma_e\le 11 \ \ \ \Delta
\Phi\ge 19 \ \ \ \Delta\tc=0 \ \ \  \DG \ge 10/9.&
\end{eqnarray*}

\medskip
\noindent{\bf Case $m\ge 6$:}
Let us denote some six $5$-neighbours of $v$ by $v_1,
\dots , v_6$.
\begin{eqnarray*}
&\text{Delete}(v) \ \ \ \text{Collect}(v_1, v_2, v_3, v_4, v_5, v_6) &\\
&\Delta |V|=7 \ \ \ \text{\bf b}=(6,0,1,0,0,0) \ \ \ \Sigma_e\le 11 \ \ \ \Delta
\Phi\ge 17 \ \ \ \Delta\tc=0 \ \ \  \DG \ge 19/18.&
\end{eqnarray*}
\end{proof}

\section{Charges on vertices of degree $5$}

In this section we show that all $5$-vertices outside bad cuts
completely discharge and none of them has positive final charge.

\begin{lema}\label{lema5}
{No ordinary $5$-vertex has positive final charge.}
\end{lema}
\begin{proof}
For a contradiction, let $v$ be such a vertex.
The vertex $v$ cannot have too many $7^+$-neighbours. If it was so, then $v$ would
completely discharge into these higher-degree vertices no matter what type they are of.
The   {neighbours of $v$}  will be denoted in accordance with Figure~\ref{fig:5}.

Our first step is to list all possible configurations around the vertex $v$.
According to Step 1 of the discharging procedure, any $5$-vertex sends at least $1/2$ into each non-triangular face it lies in.
Consequently, $v$ is contained it at most one non-triangular face. We will distinguish two cases.

\medskip

\noindent{\bf Case 1: $v$ lies in triangular faces only.}
To guarantee that we do not miss any configuration, we list them in the following order:
\begin{enumerate}
\item Configurations are ordered according to the number of $7^+$-neighbours of $v$.
The vertex $v$ has at most two such neighbours, and if it has exactly two of them,
then at least one of them must have degree exactly $7$ because each $8^+$-vertex drains at least $1/2$ from $v$.

\item The second ordering criterion is the degree of the highest-degree neighbour of $v$.
We may assume that this vertex is $v_1$ without a loss of generality (if there are two such candidates, we choose any of them).

\item Finally, if $v$ has two $7^+$-neighbours, then
by symmetry we may assume that the second vertex is either $v_2$ or $v_3$. Configurations where $v_2$ has degree more than $6$
are earlier in the ordering.
\end{enumerate}

\noindent{\bf Case 2: $v$ lies in one non-triangular face.}
The primary criterion is the same as the first criterion of Case 1, however, $v$ has now at most one $7^+$-neighbour and this neighbour has degree $7$ because otherwise it would drain all the remaining charge from $v$ (recall that $v$ sends at least $1/2$ into the non-triangular face).

The secondary criterion for this case is that
we order the configurations according to the position of the non-triangular face. If $\deg(v_1)=7$, we may
assume without a loss of generality that this is the face containing either $v_1vv_2$, $v_2vv_3$, or $v_3vv_4$. If $v_1$ has degree at most $6$,
then we may assume that $v_1vv_2$ is in the boundary of the non-triangular face.

\medskip

Listing all the configurations in this way, we obtain $14$ possible configurations of the neighbourhood of $v$.
These configurations are displayed in Figure \ref{fig:5}.
The vertices $v$, $v_1, \dots, v_5$ are \emph{denoted}, the
other vertices of $G$ are \emph{non-denoted}.
\begin{figure}[t]
\centering
\includegraphics[scale=.8]{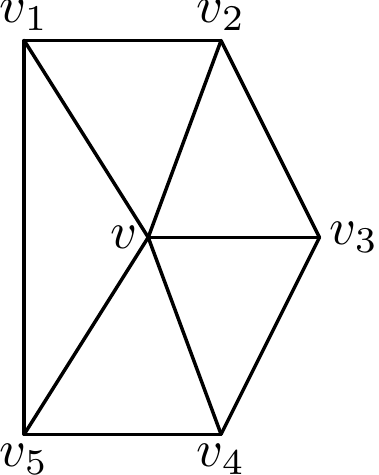}
\begin{tabular}{cccc}
\includegraphics[scale=.9]{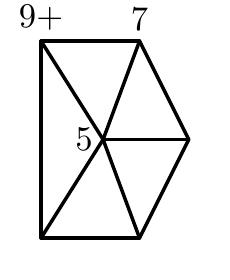}&
\includegraphics[scale=.9]{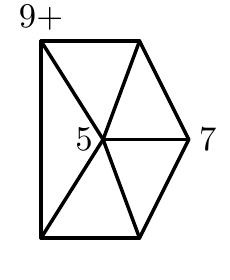}&
\includegraphics[scale=.9]{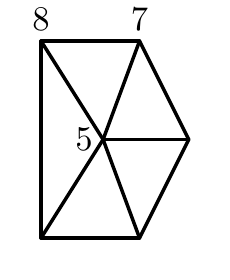}&
\includegraphics[scale=.9]{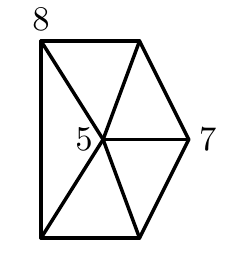} \\
Conf. $1$ & Conf. $2$ & Conf. $3$ & Conf. $4$\\
\includegraphics[scale=.9]{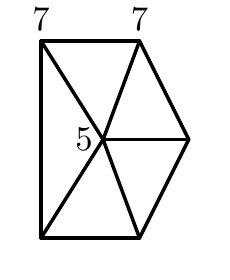}&
\includegraphics[scale=.9]{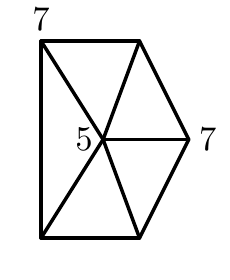}&
\includegraphics[scale=.9]{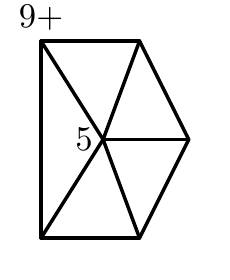}&
\includegraphics[scale=.9]{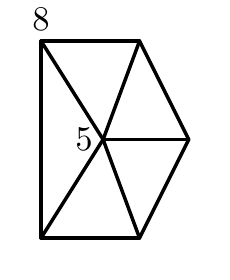} \\
Conf. $5$ & Conf. $6$ & Conf. $7$ & Conf. $8$\\
\includegraphics[scale=.9]{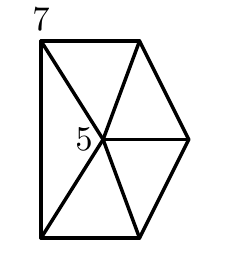}&
\includegraphics[scale=.9]{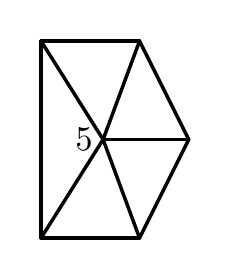}&
\includegraphics[scale=.9]{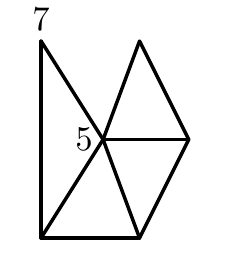}&
\includegraphics[scale=.9]{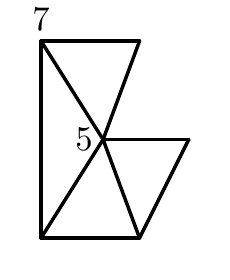}\\
Conf. $9$ & Conf. $10$ & Conf. $11$ & Conf. $12$\\
\includegraphics[scale=.9]{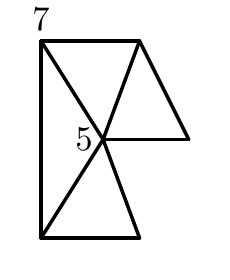}&
\includegraphics[scale=.9]{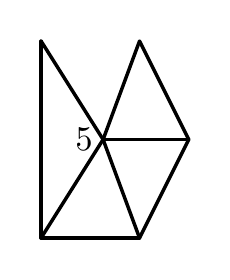}& & \\
Conf. $13$ & Conf. $14$ & &
\end{tabular}
\caption{Possible configurations around a $5$-vertex.}
\label{fig:5}
\end{figure}

\begin{claim}
\label{claim}
For all configurations except 10 and 14, the vertex $v_1$ has no non-denoted $5$-neighbour.
\end{claim}
\begin{proof} Suppose $v_1$ has a non-denoted $5$-neighbour $w$.
For all configurations except 10 and 14, we can do the following:
\begin{eqnarray*}
&\text{Delete}(v_1) \ \ \ \text{Collect}(v, v_2, v_3, v_4, v_5, w) &\\
&\Delta |V|=7 \ \ \ \text{\bf b}=(6,0,1,0,0,0) \ \ \ \Sigma_e\le 12 \ \ \ \Delta
\Phi\ge 15 \ \ \ \Delta\tc=0 \ \ \  \DG \ge 1.&
\end{eqnarray*}
\end{proof}

\medskip\noindent\underline{\emph{Configurations 1, 2, 7:}}
 {By} Claim~\ref{claim}, the vertex $v_1$ has at most three $5$-neighbours.
In Configuration 1, the type of $v_1$ is 10a, 9a, or 9b. Therefore, $v$ completely discharges into $v_1$, hence it cannot have positive final charge.
In Configurations 2 and 7, the type of $v_1$ can be 10a, 9a, 9b, or 9c. In all cases $v$ completely discharges into $v_1$.

\medskip\noindent\underline{\emph{Configuration 3:}}
  {By} Claim~\ref{claim}, the vertex $v_1$ can be of type 8a, 8b, 8c or 8d. In the first three cases, $v$ discharges completely into $v_1$.
If $v_1$ is of type 8d, $v$ discharges $9/10$ into $v_1$ and can send at least $1/3$ into $v_2$, hence $v$ cannot have positive charge.

\medskip\noindent\underline{\emph{Configuration 4:}}
If $v_2$ or $v_4$ has degree $6$, the argument used in Configuration 3 applies. Otherwise, both $v_2$ and $v_4$ are $5$-vertices.

If $v_1$ and $v_2$ have a common neighbour $w$ other than $v$, we do the following:
\begin{eqnarray*}
&\text{Delete}(w) \ \ \ \text{Collect}(v_2, v, v_5, v_1, v_4, v_3) &\\
&\Delta |V|=7 \ \ \ \text{\bf b}=(5,0,1,1,0,0) \ \ \ \Sigma_e\le 12 \ \ \ \Delta
\Phi\ge 21 \ \ \ \Delta\tc=0 \ \ \ \DG \ge 7/6.&
\end{eqnarray*}
A very similar reduction can be used if $v_1$ and $v_5$ have a common neighbour other than $v$.

If neither $v_1$ and $v_2$ nor $v_1$ and $v_5$ have a common neighbour different from $v$,
then $v_1$ lies in at least two non-triangular faces, thus it is of type 8c.
Consequently, $v$ completely discharges into $v_1$.

\medskip\noindent\underline{\emph{Configuration 8:}}
The vertex $v_1$ is of degree $8$.   {By} Claim~\ref{claim}, it has at most three $5$-neighbours.
If $v_1$ lies on at least two non-triangular faces, it is of type 8a, 8b or 8c, so $v$ discharges completely into
$v_1$.

Assume that $v_1$ lies in exactly one non-triangular face. If either $v_2$ or $v_5$ has degree $6$, then $v_1$ is of type 8a or 8b, and so $v$ completely discharges into $v_1$. We are left with the case where $\deg(v_2)=\deg(v_5)=5$.
Either a common non-denoted neighbour of $v_1$ and $v_2$ exists, or
a common non-denoted neighbour of $v_1$ and $v_5$ exists.
Say $w\neq v$ is a common non-denoted neighbour of $v_1$ and $v_2$.
\begin{eqnarray*}
&\text{Delete}(w) \ \ \ \text{Collect}(v_2, v, v_3, v_4, v_5, v_1) &\\
&\Delta |V|=7 \ \ \ \text{\bf b}=(6,0,0,1,0,0) \ \ \ \Sigma_e\le 12 \ \ \ \Delta
\Phi\ge 17 \ \ \ \Delta\tc=0 \ \ \  \DG \ge 19/18.&
\end{eqnarray*}

If $v_1$ is in no non-triangular face, then at least one of the vertices $v_2$
and $v_5$ is of degree $5$ (otherwise $v$ discharges completely into $v_1$ which would be of type 8a), say it is $v_2$.
Let $w$ be the common non-denoted neighbour of $v_2$ and $v_1$. Then
\begin{eqnarray*}
&\text{Delete}(w) \ \ \ \text{Collect}(v_2, v, v_3, v_4, v_5, v_1) &\\
&\Delta |V|=7 \ \ \ \text{\bf b}=(6,0,0,1,0,0) \ \ \ \Sigma_e\le 12 \ \ \ \Delta
\Phi\ge 17 \ \ \ \Delta\tc=0 \ \ \ \DG \ge 19/18.&
\end{eqnarray*}

\medskip\noindent\underline{\emph{Configuration 5:}}
  {By} Claim~\ref{claim}, $v_1$ cannot have a non-denoted $5$-neighbour.
Symmetrically, $v_2$ has no non-denoted $5$-neighbour.
The vertex $v$ would completely discharge into $v_1$ and $v_2$ unless $\deg(v_3)=5$ and $\deg(v_5)=5$.
Moreover, there are only triangular faces around $v_1$ and $v_2$, for otherwise one of $v_1$, $v_2$ is of type 7b and the other one 7b or 7c, so $v$
completely discharges.
Let $w\neq v$ be the common non-denoted neighbour of $v_5$ and $v_1$.
Then
\begin{eqnarray*}
&\text{Delete}(w) \ \ \ \text{Collect}(v_5, v, v_4, v_3, v_2, v_1) &\\
&\Delta |V|=7 \ \ \ \text{\bf b}=(5,0,2,0,0,0) \ \ \ \Sigma_e\le 12 \ \ \ \Delta
\Phi\ge 19 \ \ \ \Delta\tc=0 \ \ \ \DG \ge 10/9.&
\end{eqnarray*}

\medskip\noindent\underline{\emph{Configuration 6:}}
  {By} Claim~\ref{claim}, $v_1$ and $v_3$ have no non-denoted $5$-neighbours.
If $\deg(v_2)=5$ and $\deg(v_5)=5$, then let $w$ be a non-denoted neighbour of
$v_2$.
\begin{eqnarray*}
&\text{Delete}(w) \ \ \ \text{Collect}(v_2, v, v_5, v_4, v_3, v_1) &\\
&\Delta |V|=7 \ \ \ \text{\bf b}=(5,0,2,0,0,0) \ \ \ \Sigma_e\le 12 \ \ \ \Delta
\Phi\ge 19 \ \ \ \Delta\tc=0 \ \ \ \DG \ge 10/9.&
\end{eqnarray*}
The same approach works if $\deg(v_2)=5$ and $\deg(v_4)=5$.
In all the remaining cases, both of $v_1$ and $v_3$ have at most two $5$-neighbours.
This means that $v_1$ and $v_3$ are only in triangular faces, otherwise $v$
completely discharges into one
of these vertices. Moreover, either $v_2$ or $v_5$ is of degree $5$ (otherwise
$v$ discharges), say it is $v_2$.
Let $w\neq v$ be the common non-denoted neighbour of $v_2$ and $v_1$.
\begin{eqnarray*}
&\text{Delete}(w) \ \ \ \text{Collect}(v_2, v, v_1, v_5, v_4, v_3) &\\
&\Delta |V|=7 \ \ \ \text{\bf b}=(4,1,2,0,0,0) \ \ \ \Sigma_e\le 12 \ \ \ \Delta
\Phi\ge 21 \ \ \ \Delta\tc=0 \ \ \ \DG \ge 7/6.&
\end{eqnarray*}

\medskip\noindent\underline{\emph{Configuration 9:}}
If one of the vertices $v_2, \dots, v_5$ is of degree $5$, say it is $v_2$,
then let $w$ be a non-denoted neighbour of $v_2$.
\begin{eqnarray*}
&\text{Delete}(w) \ \ \ \text{Collect}(v_2, v, v_3, v_4, v_5, v_1) &\\
&\Delta |V|=7 \ \ \ \text{\bf b}=(6,0,1,0,0,0) \ \ \ \Sigma_e\le 12 \ \ \ \Delta
\Phi\ge 15 \ \ \ \Delta\tc=0 \ \ \ \DG \ge 1.&
\end{eqnarray*}
Therefore $v_2, \dots v_5$ are of degree $6$.
If some neighbour $w$ of a denoted vertex had degree $5$, we can delete $v_1$
and collect the remaining denoted vertices together with $w$ with the same calculation.
Otherwise, we may assume that no denoted vertex has a $5$-neighbour different from $v$.

If some denoted vertex $u$ is in a non-triangular face,
then $v$ discharges into $v_1$ and $u$ even though $u$ has degree $6$.
Let $w$ be the common non-denoted neighbour of $v_3$ and $v_4$.
If $\deg(w)\ge 7$, then $v$ discharges $1/5$ into $w$ via distance discharging and $4/5$ into $v_1$, thus it cannot have positive charge.
If $\deg(w)\le 6$, then
\begin{eqnarray*}
&\text{Delete}(v_3) \ \ \ \text{Collect}(v, v_4, v_2, v_5, v_1, w) &\\
&\Delta |V|=7 \ \ \ \text{\bf b}=(2,4,1,0,0,0) \ \ \ \Sigma_e\le 12 \ \ \ \Delta
\Phi\ge 23 \ \ \ \Delta\tc=0 \ \ \ \DG \ge 11/9.&
\end{eqnarray*}

\medskip\noindent\underline{\emph{Configuration 10:}}
We split the argument according to the number $m$ of denoted $5$-vertices.

\noindent{\bf Case 1: $m=6$.}
Assume that a denoted vertex, say $v_1$, has a $7^+$-neighbour $w$.
\begin{eqnarray*}
&\text{Delete}(w) \ \ \ \text{Collect}(v_1, v, v_2, v_3, v_4, v_5) &\\
&\Delta |V|=7 \ \ \ \text{\bf b}=(6,0,1,0,0,0) \ \ \ \Sigma_e\le 12 \ \ \ \Delta
\Phi\ge 15 \ \ \ \Delta\tc=0 \ \ \ \DG \ge 1.&
\end{eqnarray*}

Assume that a denoted vertex, say $v_1$, has a $6$-neighbour $w$ which is adjacent to no other denoted vertex.
\begin{eqnarray*}
&\text{Delete}(w) \ \ \ \text{Collect}(v_1, v, v_2, v_3, v_4, v_5) &\\
&\Delta |V|=7 \ \ \ \text{\bf b}=(6,1,0,0,0,0) \ \ \ \Sigma_e\le 11 \ \ \ \Delta
\Phi\ge 15 \ \ \ \Delta\tc=0 \ \ \ \DG \ge 1.&
\end{eqnarray*}

From now on, we assume that all neighbours of denoted vertices have degree at most $6$.
  {Assume}  that there exists a neighbour $w$ of a denoted vertex such that $\deg(w) = 6$. The vertex $w$ is adjacent to at least two denoted vertices. Moreover, there are no bad cuts containing $v$, and so $w$ is adjacent to exactly two denoted vertices which are also adjacent, say, $v_1$ and $v_2$.
Let $w'$ we a non-denoted neighbour of $v_4$ (it must be different from $w$ because $w$ is not adjacent to $v_4$).
If $w$ and $w'$ are adjacent, then
\begin{eqnarray*}
&\text{Delete}(w') \ \ \ \text{Collect}(v_4, v, v_1, v_2, v_3, v_5, w) &\\
&\Delta |V|=8 \ \ \ \text{\bf b}=(7,1,0,0,0,0) \ \ \ \Sigma_e\le 15 \ \
\Delta \Phi\ge 12 \ \ \Delta\tc=0^* \ \ \DG \ge 1.&
\end{eqnarray*}
If $w$ and $w'$ are not adjacent, then
\begin{eqnarray*}
&\text{Delete}(w') \ \ \ \text{Collect}(v_4, v, v_1, v_2, v_3, v_5, w) &\\
&\Delta |V|=8 \ \ \ \text{\bf b}=(7,1,0,0,0,0) \ \ \ \Sigma_e\le 14 \ \
\Delta \Phi\ge 14 \ \ \Delta\tc\ge -1^* \ \ \DG \ge 1.&
\end{eqnarray*}
[Both vertices $w$ and $w'$ may be connected to at most two denoted vertices.
Lemma~\ref{lemacomp} does not cover the necessary analysis of created tree components.
A tree component cannot be an isolated vertex because then it would be
a neighbour of $5$ removed vertices; three of them must be denoted and this forces a bad cut containing $v$.
Therefore any newly created tree component contains at least two leaves $x_1$
and $x_2$. Both these vertices must have at least four neighbours among removed vertices,
but at most two neighbours among denoted vertices (to avoid bad cuts containing $v$),
hence both $x_1$ and $x_2$ are adjacent to both $w$ and $w'$.
Thus if a tree component was created, then there cannot be an edge between $w$ and $w'$ because there would be a $K_5$-minor in $G$ (containing vertices $v$, $w$, $w'$, $x_1$, and $x_2$).
This finishes the explanation indicated by a star for the case where $w$ and $w'$ are adjacent (we have proved that no new tree components can be created).

It remains to prove that two or more new tree components cannot be created. Assume we created two tree components $S$ and $T$. We know that both $S$ and $T$ contain two leaves connected to both $w$ and $w'$ and having at least two denoted neighbours; let those leaves be $s_1$ and $s_2$ for $S$ and $t_1$ and $t_2$ for $T$. The paths $ws_1w'$, $ws_2w'$, $wt_1w'$, $wt_2w'$ divide the plane into four regions. One of those regions contains $v$, say it is the one with boundary $ws_1w's_2$.
But then $t_1$ or $t_2$ is separated from all the denoted vertices and has no neighbours among them though it should have two due to its definition.]

We are left with the case where all neighbours of denoted vertices have
degree $5$.
This is together at least $11$ vertices: there are exactly two edges joining a denoted vertex to non-denoted vertices; on the other hand, a non-denoted vertex has at most two denoted neighbours since we are avoiding bad cuts containing $v$, and so there are at least five non-denoted vertices.
These $11$ vertices must have another neighbour
(otherwise, we would be able to collect the whole graph because any planar graph with at most $11$ vertices contains a $4$-vertex).
Deleting that neighbour and collecting the $11$ vertices itself decreases $\Gamma$ by at least $1$.
Each created tree component must have at least $5$ neighbours among removed vertices, therefore $\Delta \Phi$ is at least five times the number of newly created tree components, and so $ \frac 1{36}\Delta\Phi +\frac 1{18}\Delta\tc\ge 0$. Altogether, $\Gamma$ decreases by at least $1$.

\noindent{\bf Case 2: $m=5$.}
Let $v_1$ be the $6$-neighbour of $v$. If a denoted $5$-vertex, say $v_2$, has a non-denoted $6^+$-neighbour $w$, then
\begin{eqnarray*}
&\text{Delete}(w) \ \ \ \text{Collect}(v_1, v, v_2, v_3, v_4, v_5) &\\
&\Delta |V|=7 \ \ \ \text{\bf b}=(5,2,0,0,0,0) \ \ \ \Sigma_e\le 12 \ \ \ \Delta
\Phi\ge 15 \ \ \ \Delta\tc=0 \ \ \ \DG \ge 1.&
\end{eqnarray*}

Let us call the vertices $v_2$, $v_3$, $v_4$, $v_5$ {\it red} and their non-denoted neighbours {\it blue}.
We are left with the situation where all red and blue vertices have degree $5$. Each red vertex has two blue neighbours, so there are potentially eight blue vertices. A blue vertex can be a neighbour of two red vertices only if those red vertices are $v_2$ and $v_3$, $v_3$ and $v_4$, or $v_4$ and $v_5$. Moreover, each of the listed pairs has at most one common blue neighbour. Consequently, there are at least five blue vertices.

If there are at least six blue vertices, we can delete one of them, then subsequently collect all red vertices, then collect the remaining blue vertices, and finally collect $v$ and $v_1$. Altogether, we can remove $12$ vertices while deleting just one of them. Removing those $12$ vertices itself decreases $\Gamma$ by at least $1$.
Each created tree component must have at least $5$ neighbours among removed vertices, therefore $\Delta \Phi\ge |\Delta\tc|$. Hence $\Gamma$ decreases by at least $1$.

If there are exactly five blue vertices and there exists a non-denoted neighbour $w$ of a blue vertex which is not blue, we can delete $w$ and then subsequently collect all blue and denoted vertices, removing $12$ vertices from $G$. The argument from the previous paragraph applies in this case, too.

If the five blue vertices have no additional neighbour, then $v_1$ is a cut-vertex (due to Euler's formula, there are exactly two non-denoted neighbours of $v_1$ different from blue  vertices). Consequently, $v_1v_2v$ is a bad $C_3$-cut; a contradiction.

\noindent{\bf Case 3: $m\le 4$.}
Assume that some neighbour of $v$, say $v_1$, has degree $5$. Let $w$ be a non-denoted neighbour of $v_1$.
\begin{eqnarray*}
&\text{Delete}(w) \ \ \ \text{Collect}(v_1, v, v_2, v_3, v_4, v_5) &\\
&\Delta |V|=7 \ \ \ \text{\bf b}=(5,2,0,0,0,0) \ \ \ \Sigma_e\le 12 \ \ \ \Delta
\Phi\ge 15 \ \ \ \Delta\tc=0 \ \ \ \DG \ge 1.&
\end{eqnarray*}
Otherwise, all neighbours of $v$ have degree $6$.
Next, assume that a neighbour of $v$, say $v_1$, has a non-denoted neighbour $w$ of
degree $5$.
\begin{eqnarray*}
&\text{Delete}(v_1) \ \ \ \text{Collect}(w, v, v_2, v_3, v_4, v_5) &\\
&\Delta |V|=7 \ \ \ \text{\bf b}=(2,5,0,0,0,0) \ \ \ \Sigma_e\le 12 \ \ \ \Delta
\Phi\ge 21 \ \ \ \Delta\tc=0 \ \ \ \DG \ge 7/6.&
\end{eqnarray*}
Otherwise, all non-denoted neighbours of denoted vertices have degree at least $6$, and thus denoted vertices have at most one $5$-neighbour.

Let us look at the edge $v_1v_2$. If $v_1$ and $v_2$ are only on triangular
faces, then they have a common neighbour $w$.
The vertex $w$ has degree at least $7$, otherwise we can do the following:
\begin{eqnarray*}
&\text{Delete}(v_1) \ \ \ \text{Collect}(v, v_2, v_3, v_4, v_5, w) &\\
&\Delta |V|=7 \ \ \ \text{\bf b}=(2,5,0,0,0,0) \ \ \ \Sigma_e\le 12 \ \ \ \Delta
\Phi\ge 21 \ \ \ \Delta\tc=0 \ \ \ \DG \ge 7/6.&
\end{eqnarray*}
In that case we can discharge $1/5$ from $v$ to $w$ via distance discharging.
We repeat this argument for the edges $v_2v_3$, $v_3v_4$, $v_4v_5$, $v_5v_1$.
If all vertices $v_i$, for $i\in \{ 1, 2, 3, 4, 5\}$, are only in
triangular faces, then $v$ discharges.
If some vertex $v_i$, for $i\in \{ 1, 2, 3, 4, 5\}$, is in a non-triangular
face, then it is of type 6a and $v$ discharges $2/5$ into this vertex.
This compensates for the inability to do
distance discharging to neighbours of $v_i$. Consequently, in any case $v$ cannot have positive charge.

\medskip\noindent\underline{\emph{Configuration 11:}}
The vertex $v$ discharges into the non-triangular face and into $v_1$,
which has no more than two $5$-neighbours.

\medskip\noindent\underline{\emph{Configurations 12 and 13:}}
If both $v_2$ and $v_5$ have degree $5$, then let $w$ be a non-denoted neighbour
of $v_2$.
\begin{eqnarray*}
&\text{Delete}(w) \ \ \ \text{Collect}(v_2, v, v_5, v_3, v_4, v_1) &\\
&\Delta |V|=7 \ \ \ \text{\bf b}=(6,0,1,0,0,0) \ \ \ \Sigma_e\le 11 \ \ \ \Delta
\Phi\ge 17 \ \ \ \Delta\tc=0 \ \ \ \DG \ge 19/18.&
\end{eqnarray*}
Otherwise, at least one of the vertices $v_2$, $v_5$ has degree $6$.
If the other of them also had degree $6$, then $v_1$ would be of type 7a and $v$ would completely discharge.
Thus we are left with the case where $\{\deg(v_2), \deg(v_5)\} = \{5, 6\}$.

If $v_5$ has degree $5$, then there is a common non-denoted neighbour $w$ of $v_1$ and $v_5$ because $v_1$
cannot be in a non-triangular face (if it was, it would be of type 7b and would discharge). Thus we can do
\begin{eqnarray*}
&\text{Delete}(w) \ \ \ \text{Collect}(v_5, v, v_1, v_2, v_4, v_3) &\\
&\Delta |V|=7 \ \ \ \text{\bf b}=(6,0,1,0,0,0) \ \ \ \Sigma_e\le 11 \ \ \ \Delta
\Phi\ge 17 \ \ \ \Delta\tc=0 \ \ \ \DG \ge 19/18.&
\end{eqnarray*}

If $v_2$ has degree $5$, the situation is symmetric to the previous one for Configuration 12.
For Configuration 13, we delete the common non-denoted neighbour $w$ of $v_2$ and $v_1$ and collect the remaining denoted vertices in the order $v_2$, $v$, $v_1$, $v_5$, $v_4$, $v_3$. The calculations are the same.

\bigskip\medskip\noindent\underline{\emph{Configuration 14:}}
If $\deg(v_1)=5$, then let $w$ be a neighbour of $v_1$ that is not adjacent to
any denoted vertex (one such neighbour lies on the non-triangular face containing $v$).
\begin{eqnarray*}
&\text{Delete}(w) \ \ \ \text{Collect}(v_1, v, v_2, v_3, v_4, v_5) &\\
&\Delta |V|=7 \ \ \ \text{\bf b}=(7,0,0,0,0,0) \ \ \ \Sigma_e\le 10 \ \ \ \Delta
\Phi\ge 15 \ \ \ \Delta\tc=0 \ \ \ \DG \ge 1.&
\end{eqnarray*}
Therefore $\deg(v_1)=6$. If any other denoted vertex besides $v$ has degree $5$,
say it is $v_3$, then let $w$ be a neighbour of $v_3$.
\begin{eqnarray*}
&\text{Delete}(w) \ \ \ \text{Collect}(v_3, v, v_2, v_4, v_5, v_1) &\\
&\Delta |V|=7 \ \ \ \text{\bf b}=(6,1,0,0,0,0) \ \ \ \Sigma_e\le 11 \ \ \ \Delta
\Phi\ge 15 \ \ \ \Delta\tc=0 \ \ \ \DG \ge 1.&
\end{eqnarray*}
If any non-denoted neighbour of a denoted vertex has degree $5$, say it is
$v_1$, then let $w$ be that neighbour.
\begin{eqnarray*}
&\text{Delete}(w) \ \ \ \text{Collect}(v_1, v, v_5, v_4, v_3, v_2) &\\
&\Delta |V|=7 \ \ \ \text{\bf b}=(2,5,0,0,0,0) \ \ \ \Sigma_e\le 11 \ \ \ \Delta
\Phi\ge 23 \ \ \ \Delta\tc=0 \ \ \ \DG \ge 11/9.&
\end{eqnarray*}
Otherwise, $v_1$ and $v_2$ are of type 6a, so each of them can get $2/5$ from $v$.
Since $v_1$ sends at least $1/2$ into the non-triangular face, it cannot have positive charge.
\end{proof}

\begin{proof}[Proof of Theorem \ref{thm:main2}]
Lemma \ref{lemanoext} and Lemma \ref{lemaexample}
show that a minimal counterexample to Theorem
\ref{thm:main2} contains an ordinary vertex with positive charge.
Lemma \ref{lema8}, Lemma \ref{lema6}, and Lemma \ref{lema5} say that no such
vertex exists, which is a contradiction.
\end{proof}

Observe that in most of the cases, $\DG > 1$. If $\DG > 1$ in all cases, we could have 
defined $\Gamma(G)$ as $$\Gamma(G) = {1\over 12}|V(G)| +
\lambda\Phi(G)  + 2 \lambda\tc(G)$$ for some $\lambda < 1/36$, which would lead to an improvement of 
Theorem \ref{thm:main2}. 
The most problematic case which hinders further improvement is the case when a
$5$-vertex is surrounded
by five $5$-vertices. One of these configurations is in Figure~\ref{pconf}.

\begin{figure}
\begin{center}
\includegraphics{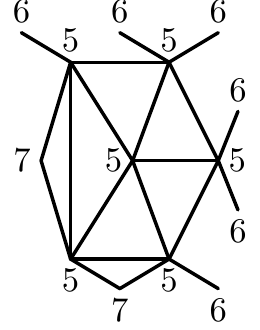}
\end{center}
\caption{A problematic configuration. No extra edges between depicted vertices
exist.}\label{pconf}
\end{figure}

\end{document}